\documentclass{article}

\PassOptionsToPackage{numbers, compress}{natbib}

\usepackage[final]{nips_2018}
\usepackage{Definitions}

\usepackage[utf8]{inputenc} 
\usepackage[T1]{fontenc}    
\usepackage{hyperref}       
\usepackage{url}            
\usepackage{booktabs}       
\usepackage{amsfonts}       
\usepackage{nicefrac}       
\usepackage{microtype}      
\usepackage{graphicx}
\usepackage{subfig}
\usepackage{xcolor}
\usepackage{paralist}
\usepackage{mathtools}
\usepackage{enumitem}
\usepackage{algorithm}

\usepackage{algpseudocode}

\newcommand*{\xhdr}[1]{\noindent{{\bf #1.}}}

\title{Stochastic Optimal Control of \\ Epidemic Processes in Networks}

\author{
Lars Lorch \\ 
Harvard University\\
\texttt{lorch@college.harvard.edu} \\
\And
Abir De \\
MPI-SWS\\
\texttt{ade@mpi-sws.org} \\
\And
Samir Bhatt \\
Imperial College London \\
\texttt{s.bhatt@imperial.ac.uk} \\
\And
William Trouleau \\
EPFL \\
\texttt{william.trouleau@epfl.ch} \\
\And
Utkarsh Upadhyay \\
MPI-SWS\\
\texttt{utkarshu@mpi-sws.org} \\
\And
Manuel Gomez-Rodriguez \\
MPI-SWS\\
\texttt{manuelgr@mpi-sws.org} \\
}

\graphicspath{{./FIG/}}

\begin{document}

\maketitle


\vspace*{-8mm}
\section{Introduction}
\label{sec:introduction}
\vspace{-1mm}

%
Most of the early work on epidemic processes~\cite{model1,model2,model5,model3,model4} focused primarily on developing models for general population dynamics
rather than the state of any given individual in the population.
More recently, there has been research on modeling individual dynamics of epidemics in networks
~\cite{ahn2013global,cator2012second,chakrabarti2008epidemic,van2011n,van2009virus},
which has been fueled by the success of network science in explaining other real world processes~\cite{net1,net2}. 
However, most of these models ultimately resort to mean-field theory ~\cite{cator2012second}, which precludes them from characterizing the exact state of each individual.
As a consequence, they are unable to benefit from the increasing availability of fine-grained data about disease outbreaks~\cite{routledge2018estimating}.

In this context, an orthogonal line of research aims to design control strategies to mitigate the spread of epidemics in networks using various optimization 
techniques~\cite{drakopoulos2014efficient, ganesh2005effect, hansen2011optimal,khanafer2014optimal,scaman2016suppressing,tong2012gelling,wickwire1975optimal,zhang2011risk}. 
However, these approaches share one or more of the following shortcomings, which reduce its applicability in practice: 
(i) they leverage simple deterministic models of population dynamics and are thus unable to accurately assess the effect of individual treatments; 
(ii) they propose off-line strategies, often based on graph theoretic quantities, which do not account for the infection and recovery events of a specific epidemic over time, and as a consequence, achieve suboptimal performance; 
and
(iii) they assume that individuals can go on and off of treatment multiple times instantaneously or vary recovery rates arbitrarily, which is not realistic in a real-world setting.

In this work, we approach the development of models and control strategies of epidemic processes from the perspective of marked temporal point processes (MTPPs)~\cite{snyder2012random} 
and stochastic optimal control of stochastic differential equations (SDEs) with jumps~\cite{hanson2007}.
In contrast to previous work, this novel perspective is particularly well-suited to make use of fine-grained data on disease outbreaks and lets us overcome the shortcomings of current control strategies.
%
%
More specifically, we focus on susceptible-infected-susceptible (SIS) epidemic processes and represent the times when each individual becomes infected, recovered, or treated using MTPPs. We then exploit an alternative representation using SDEs with jumps, which have recently been introduced in the machine learning literature~\cite{kim2018leveraging, zarezade2018steering, zarezade2017redqueen}, 
to cast the design of control strategies of epidemic processes as a stochastic optimal control problem.
%
%
%
The solution to this problem provides us with treatment intensities to determine who to treat and when to do so to minimize the amount of infected individuals over time.
Preliminary experiments with synthetic data show that our control strategy consistently outperforms several alternatives.
Looking into the future, we believe our methodology provides a promising step towards the development of practical data-driven control strategies of epidemic processes.


\vspace{-1mm}
\section{Problem Formulation}
\label{sec:formulation}
\vspace{-1mm}

In this section, we first present our MTPP model of susceptible-infected-susceptible (SIS) epidemic processes, and then use an alternative representation
of MTPPs using SDEs with jumps to cast the design of control strategies as a stochastic optimal control problem.

\xhdr{A MTPP model of SIS epidemic processes}
Given an undirected network $\Gcal=(\Vcal,\Ecal)$ where nodes represent individuals and edges represent contacts, we represent the times when each 
node gets infected, recovered, or treated using a collection of counting processes $\Yb(t)$, $\Wb(t)$ and $\Nb(t)$, where the $i$-th entries, $Y_i(t)$, $W_i(t)$, and $N_i(t)$, count the number of times node $i$ got infected, recovered, and treated, respectively, by time $t$.
In addition, we track the current state of a node and its neighbors using a collection of state variables $\Xb(t)$, $\Hb(t)$, $\Zb(t)=\Ab^\top \Xb(t)$ and 
$\Mb(t)=\Ab^\top \Hb(t)$, where $X_i(t)=1$ ($H_i(t) = 1$) if node $i$ is infected (under treatment) at time $t$, and $X_i(t)=0$ ($H_i(t) = 0$) otherwise. 
$\Ab$ denotes the adjacency matrix with $A_{ij} = 1$ if node $i$ is connected to node $j$ and $A_{ij} = 0$ otherwise. Moreover, $\Ab_{i*}$ is defined as the $i$-th row, $\Ab_{*i}$ as the $i$-th column of $\Ab$. Thus, $\Zb(t)$ ($\Mb(t)$) denotes the number of nodes infected (under treatment) in the neighborhood $\Ncal(i)$ of node $i$.

We then characterize the SIS dynamics of the state variables in terms of the counting processes using the following stochastic differential equations (SDEs) 
with jumps:
\begin{equation} \label{eq:sis-vars}
\begin{split}
dX_i(t) &= dY_i(t) - dW_i(t)   \quad\quad\quad\quad \quad\:
dZ_i(t) = \Ab_{*i}^\top \bm{dY}(t) - \Ab_{*i}^\top  \bm{dW}(t) \\
dH_i(t) &= dN_i(t) - H_i(t) \cdot dW_i(t)  \quad \quad
dM_i(t)  = \Ab_{*i}^\top \bm{dN}(t) - \Ab_{*i}^\top (\Hb(t) \odot \bm{dW}(t))
\end{split}
\end{equation}
where ``$\odot$'' denotes element-wise vector multiplication, $\Xb(0) = \Xb_0$, $\Zb(0) = \Ab^\top \Xb_0$, and $\Hb(0) = \Mb(0) = \mathbf{0}$.
Inspired by previous work~\cite{scaman2016suppressing}, we characterize the above counting processes using the following intensity functions or
\emph{rates}:
\begin{equation}\label{eq:sis-rates}
\begin{split}
\EE\left[ dY_i(t) | \mathcal{H}(t) \right] &= (1-X_i(t)) (\beta Z_i(t) - \gamma M_i(t)) dt \\
\EE\left[ dW_i(t) | \mathcal{H}(t) \right] &= X_i(t) ( \delta + \rho H_i(t)) dt \\
\EE\left[ dN_i(t) | \mathcal{H}(t) \right] &= \lambda_i (t) X_i(t) (1 - H_i(t)) dt,
\end{split}
\end{equation}
%
%
where $\beta$, $\gamma$, $\delta$ and $\rho$ are disease specific parameters and $\Hcal(t)$ denotes the history of all counting processes up to time $t$. Under these definitions, a node can recover only if it is currently infected, get infected only if it is healthy, and undergo treatment only if it is infected and not under treatment yet.
Moreover, if a node gets infected, the infection rate of its healthy neighbors increases by $\beta$ units. If the node undergoes treatment, the infection rate of its healthy 
neighbors decreases by $\gamma \leq \beta$ units and its own recovery rate increases by $\rho$ units. If the node gets on treatment, the treatment continues until it recovers. 
Our goal is thus to find the control signals $\lambda_i(t) \geq 0$ that determine the treatment rates, \ie, the times when nodes start to undergo treatment to minimize the epidemic outbreak.

\xhdr{Control of epidemic processes as a stochastic optimal control problem}
Given an undirected network $\Gcal=(\Vcal,\Ecal)$ with node states $\Xb(t)$, $\Hb(t)$, $\Zb(t)$, and $\Mb(t)$, our goal is to find the non-negative 
control signals $\Lambdab(t) = (\lambda_i(t))_{i \in \Vcal}$ that minimize the expected discounted value of a particular loss function $\ell(\Lambdab(t), \Xb(t), \Hb(t), \Zb(t))$ 
over an infinite time horizon $[t_0, \infty)$, \ie,
\begin{align} \label{eq:minimization}
\underset{\Lambdab(t_0, \infty)}{\text{minimize}} & \quad \EE_{(\Yb, \Wb, \Nb)(t_0, \infty)}\left[ \int_{t_0}^{\infty} e^{-\eta t} \; \ell(\Lambdab(t), \Xb(t), \Hb(t), \Zb(t)) dt \right] \nonumber \\
\text{subject to} & \quad \lambda_i(t) \geq 0 \quad \forall i \in \Vcal, \forall t \in (t_0, \infty],
\end{align}
where $\Lambdab(t_0, \infty)$ denotes the values of the control signal over the infinite time horizon, $\eta \geq 0$ is the discount rate, and the expectation is taken over all possible realizations
of the counting processes from $t_0$ to $\infty$.
Moreover, we will consider the following functional form for the loss function:
\begin{align}
\ell(\Lambdab(t), \Xb(t)) &= \frac{1}{2} \Lambdab^\top(t) \Qb_{\Lambda} \Lambdab(t) +   \qb_{\Xb} \Xb(t) \label{eq:loss}
\end{align}
where $\Qb_{\Lambdab}$ is a given diagonal matrix with $(\Qb_{\Lambdab})_{ii} \geq 0$ and $\qb_{\Xb}$ is a given vector with $(\qb_{\Xb})_{i} \geq 0$. These parameters allow us to trade off how costly treatment control $\Lambdab(t)$ and consequently treatments $\Nb(t)$ are compared to the cost incurred from infected nodes $\Xb(t)$.


%

\vspace{-2mm}
\section{Stochastic Optimal Control Algorithm}
\label{sec:algorithm}
\vspace{-2mm}

In this section, we first define a novel optimal cost-to-go function that accounts for the unique aspects of our problem, and then show that the Bellman'{}s principle of optimality holds. Exploiting this principle, we find the optimal solution using the corresponding Hamilton-Jacobi-Bellman (HJB) equation.

\vspace{-1mm}
\begin{definition}
The optimal cost-to-go function is defined as the minimum of the expected value of the cost of going from state $(\Xb(t), \Hb(t), \Zb(t), \Mb(t))$ at time $t$ to the final state at $t \to \infty$.
\begin{equation} \label{eq:def-cost-to-go}
\hspace*{-0.5cm}J(\Xb(t), \Hb(t), \Zb(t), \Mb(t), t)= \hspace*{-0.2cm}\underset{\Lambdab(t, \infty)}{\min} \EE_{(\Yb, \Wb, \Nb)(t, \infty)} \hspace*{-0.05cm} \left[\hspace*{-0.1cm}  \int_{t}^{\infty} \hspace*{-0.3cm} e^{-\eta \tau} \ell(\Lambdab(\tau), \Xb(\tau), \Hb(\tau), \Zb(\tau)) d\tau \right] \hspace*{-0.3cm} 
\end{equation} 
where the expectation is taken over all realizations of the underlying counting processes $\Yb$, $\Wb$ and $\Nb$ over the time interval $(t,\infty)$.
\end{definition}
To find the optimal control signal $\Lambdab(t, \infty)$ and the optimal cost-to-go $J$, we break the problem into several smaller subproblems using
the following Lemma, which resorts to Bellman'{}s principle of optimality (proven in Appendix~\ref{app:lemma:diff}):
\vspace{-1mm}
\begin{lemma}\label{lemma:diff}
The optimal cost-to-go $J$ defined in Eq.~\ref{eq:def-cost-to-go} satisfies the following equation:
\begin{align}
\hspace*{-0.5cm}0 = \hspace*{-0.2cm} \min_{\Lambdab(t, t+dt]}
\hspace*{-0.1cm}\left\{ \EE_{(\Yb, \Wb, \Nb)(t, t+dt]}\hspace*{-0.1cm}\left[ dJ(\Xb(t), \Hb(t), \Zb(t), \Mb(t), t) \right] + e^{-\eta t} \ell(\Lambdab(t), \Xb(t), \Hb(t))dt  \right\} \hspace*{-0.4cm} \label{eq:ss}
\end{align} \vspace{-3mm}
\end{lemma}
To proceed in the infinite horizon setting, we introduce the functional
\begin{align} \label{eq:v}
&V(\Xb(t), \Hb(t), \Zb(t), \Mb(t)) = e^{\eta t} J(\Xb(t), \Hb(t), \Zb(t), \Mb(t), t) \\
 &\qquad = \textstyle \underset{\Lambdab(t, \infty)}{\text{min}} \EE_{(\Yb, \Wb, \Nb)(t, \infty)}\left[ \int_{t}^{\infty} e^{-\eta (\tau - t)} \; \ell(\Lambdab(\tau), \Xb(\tau), \Hb(\tau), \Zb(\tau)) d\tau \right] \nn 
\end{align}
\ie, the time-independent cost functional, which only depends on the elapsed time and initial states \cite{kamienschwartz}. 
Then, using It\^{o}'{}s  calculus~\cite{hanson2007}, the optimal control signal $\Lambdab^{*}(t)$ is given by the following Lemma (proven in Appendices \ref{app:lemma:aux} and ~\ref{app:lemma:hjb}):

\vspace{-1mm}
\begin{lemma} \label{lemma:hjb}
Given Eqs.~\ref{eq:sis-vars},~\ref{eq:ss} and~\ref{eq:v}, the optimal control intensity is:
\begin{equation*}
\Lambdab^*(t) = - \Qb_{\Lambdab}^{-1} \diag \big(\mathbf{1} - \Hb(t) \big) \, \Delta_{V}^N  \, \Xb(t) 
\end{equation*}
where the functional $V$ needs to satisfy the following partial differential equation (PDE):
\begin{align}\label{eq:pde}
\begin{split}
0 &= - \eta V(\Xb(t), \Hb(t), \Zb(t), \Mb(t)) +  (\mathbf{1}-\Xb(t))^\top \Delta_{V}^Y \, (\beta\Zb(t) - \gamma \Mb(t))\\
&\quad+ ( \delta \mathbf{1} + \rho\Hb)^\top  \Delta_{V}^W  \,\Xb(t) - \frac{1}{2} \Xb(t)^\top \big( \Delta_{V}^N \big)^2   \diag\big( \mathbf{1} - \Hb(t) \big)^2 \Qb_{\Lambdab}^{-1}  \: \Xb(t)  + \qb_{\Xb} \Xb(t),
\end{split}
\end{align}
and $ \Delta_{V}^Y,  \Delta_{V}^W,  \Delta_{V}^N$ are diagonal matrices  defined for notational simplicity in appendix \ref{app:lemma:hjb}.
%
\end{lemma}

%
%
To compute $\Lambdab^{*}(t)$ in the above, we need to solve the PDE in Eq.~\ref{eq:pde}. Fortunately, the following theorem provides us with a solution:
\vspace{-1mm}
\begin{theorem}\label{thm:finalxx}
There exist $\bb$, $\cb$, $\db$ and $\fb$ such that $V(\Xb(t), \Hb(t), \Zb(t), \Mb(t)) = \bb^\top \Xb(t) + \cb^\top \Hb(t) + \db^\top \Zb(t) + \fb^\top \Mb(t)$ is a valid
solution to the PDE in Eq.~\ref{eq:pde}. Then, the optimal control signal $\Lambdab^{*}(t) = (\lambda^*_i(t))_{i \in \Vcal}$ is given by:
\begin{align*}
\hspace*{-0.13cm}
\lambda_i^*(t) \hspace*{-0.03cm}= \hspace*{-0.03cm}\begin{cases}
\textstyle 0 ~~ \text{if } X_i(t) = 0 \\
\textstyle - \frac{1}{K^{(\text{I})}}  ( 1 \hspace*{-0.04cm} - \hspace*{-0.04cm} H_i(t) )  Q_{\Lambda, i}^{-1} \Big [ K^{(\text{II})}_{\Lambda, i}  - \hspace*{-0.02cm}\sqrt{  2 K^{(\text{I})} Q_{\Lambda, i} \big ( K^{(\text{III})} \Ab_{i*} \db \hspace*{-0.03cm} + \hspace*{-0.03cm} K^{(\text{IV})}_{X, i} \big ) \hspace*{-0.07cm}+\hspace*{-0.07cm} \big(K^{(\text{II})}_{\Lambda, i}\big)^2 } \Big ] ~ \text{if } X_i(t) \hspace*{-0.05cm} = \hspace*{-0.05cm}1
\end{cases}
\end{align*}
where $\Kb^{(.)}$ are constants depending on model parameters, and whenever $\Xb(t)$ changes, $\db$ is recomputed online by solving the following linear program (LP)
%
%
\begin{align}
\min_{\db_{(0)}} \quad \big| \Ab_{(10)} \db_{(0)} + \tfrac{1}{K^{(\text{III})} } \bm{K}^{(\text{IV})}_{X, (1)}  \big|_1 \quad \text{s.t.} & \quad \Ab_{(10)} \db_{(0)} \geq - \tfrac{1}{K^{(\text{III})} } \bm{K}^{(\text{IV})}_{X, (1)} \label{eq:policy-optimization}
\end{align}
where $\db_{(0)}$ is the vector $\db$ containing only indices where $X_i(t) = 0$, and $\Ab_{(10)}$ contains the $i$-th rows of $\Ab$ iff $X_i(t) = 1$ and $i$-th columns of $\Ab$ iff $X_i(t) = 0$. $\db_{(1)} = \mathbf{0}$.
Refer to Appendix ~\ref{app:thm:finalxx}~ for details.
%
%
%
%
\end{theorem}
Finally, given the above optimal control signal $\Lambdab^*(t)$, the online algorithm to determine which nodes to treat over time based on the control signal directly follows and is summarized in Appendix~\ref{app:algorithm}.


\vspace{-0mm}
\section{Experiments}
\label{sec:expt}
\vspace{-0mm}

\xhdr{Experimental setup}
We use the network $\Gcal = (\Vcal, \Ecal)$ of the contiguous states of the United States borders (49 nodes, 107 edges) and experiment with different
settings for the model parameters $\beta$, $\gamma$, $\delta$, and $\rho$ and control parameters $\Qb_\Lambda$ and $\qb_X$, which trade off how 
much we value treatments and infections. A low $\qb_X$ relative to $\Qb_\Lambda$ indicates that the effort of treating infected nodes is considered expensive relative to the cost incurred from having nodes infected. Conversely, a high $\qb_X$ would indicate that we value being in a healthy state more than reducing the effort of treating infected nodes.  

We use the number of infected nodes, $\mathbf{1}^{\top} \Xb(t)$, and the \emph{total infection coverage}, $\int_{t_0}^{t_f} \mathbf{X} (t) dt$, as measures of performance, 
and we compare our stochastic optimal control algorithm to several alternatives:
\vspace{-1mm}
\begin{itemize}[noitemsep,nolistsep,leftmargin=0.7cm]
\item[--] \emph{Trivial policy [T]}: $\lambda_i(t)$ is constant and equal for all infected nodes not yet under treatment.
\vspace{1mm}
\item[--] \emph{Front-loaded trivial policy [T-FL]}: $\lambda_i(t) = ||\Lambdab^{*}(t)||_{\infty} \II( \mathbf{1}^{\top} \Nb(t) \leq \EE[\mathbf{1}^{\top} \Nb^{*}(t_f)] )$ for all infected nodes not yet under treatment, 
where $\Lambdab^{*}(t)$ and $\Nb^*(t)$ is the optimal control signal and number of treatments by our optimal control algorithm, respectively.
\vspace{1mm}
\item[--] \emph{Most-neighbors degree policy [MN]}: $\lambda_i(t) \propto \text{deg}(i)$  for all infected nodes not yet under treatment, where $\text{deg}(i)$ is the degree of node $i$. This policy prioritizes central nodes in the network.
\vspace{1mm}
\item[--] \emph{Front-loaded most-neighbors degree policy [MN-FL] }: $\lambda_i(t) \propto \text{deg}(i) ||\Lambdab^{*}(t)||_{\infty} \II( \mathbf{1}^{\top} \Nb(t) \leq \EE[\mathbf{1}^{\top} \Nb^{*}(t_f)] )$ 
for all infected nodes not yet under treatment.
\vspace{1mm}
\item[--] \emph{Least-neighbors degree policy [LN]}: $\lambda_i(t) \propto ( \max_{j \in \Vcal}  \{\text{deg}(j)\} - \text{deg}(i) + 1)$ for all infected nodes not yet under treatment. This policy prioritizes nodes in the periphery of the network.
\vspace{1mm}
\item[--] \emph{Front-loaded least-neighbors degree policy [LN-FL] } $\lambda_i(t) \propto ( \max_{j \in \Vcal}  \{\text{deg}(j)\} - \text{deg}(i) + 1) ||\Lambdab^{*}(t)||_{\infty} \II( \mathbf{1}^{\top} \Nb(t) \leq \EE[\mathbf{1}^{\top} \Nb^{*}(t_f)] )$
for all infected nodes not yet under treatment.
\vspace{1mm}
\item[--] \emph{Largest reduction in spectral radius policy [LRSR]}: $\lambda_i(t)$ is proportional to which node'{}s removal leads to the greatest decrease in spectral radius of $\mathbf{A}$~\cite{scaman2016suppressing, tong2012gelling}.
\end{itemize}
By definition, the front-loaded policies perform approximately the same number of treatments as our optimal control algorithm by the end of the simulation. For the trivial policy (T), the most-neighbors degree policy (MN), the least-neighbors degree policy (LN), and the largest reduction in spectral radius policy (LRSR), we rescale the intensities such that they have performed approximately ($\pm5$\%) the same number of treatments $\mathbf{1}^{\top} \Nb(t_f)$ as our optimal control algorithm by the end of the simulation.
Moreover, for each parameter setting, we performed $50$ independent simulations, and in each simulation, we picked ten nodes at random as initial infections at time $t = 0$.

\xhdr{Results} Figure~\ref{fig:results} summarizes the results for one set of model parameters\footnote{\scriptsize We obtained qualitatively similar results for other model parameter settings.}, which show that our stochastic optimal control algorithm outperforms all other policies in reducing the spread of the epidemic. 
In particular, it becomes apparent that especially for high $\qb_X$, our algorithm neglects the cost of control in optimizing the objective function and heals nodes in a more effective way than all other policies, as shown in the difference of performance for higher $\qb_X$.
\begin{figure}[t]
\centering
\includegraphics[width=0.35\textwidth]{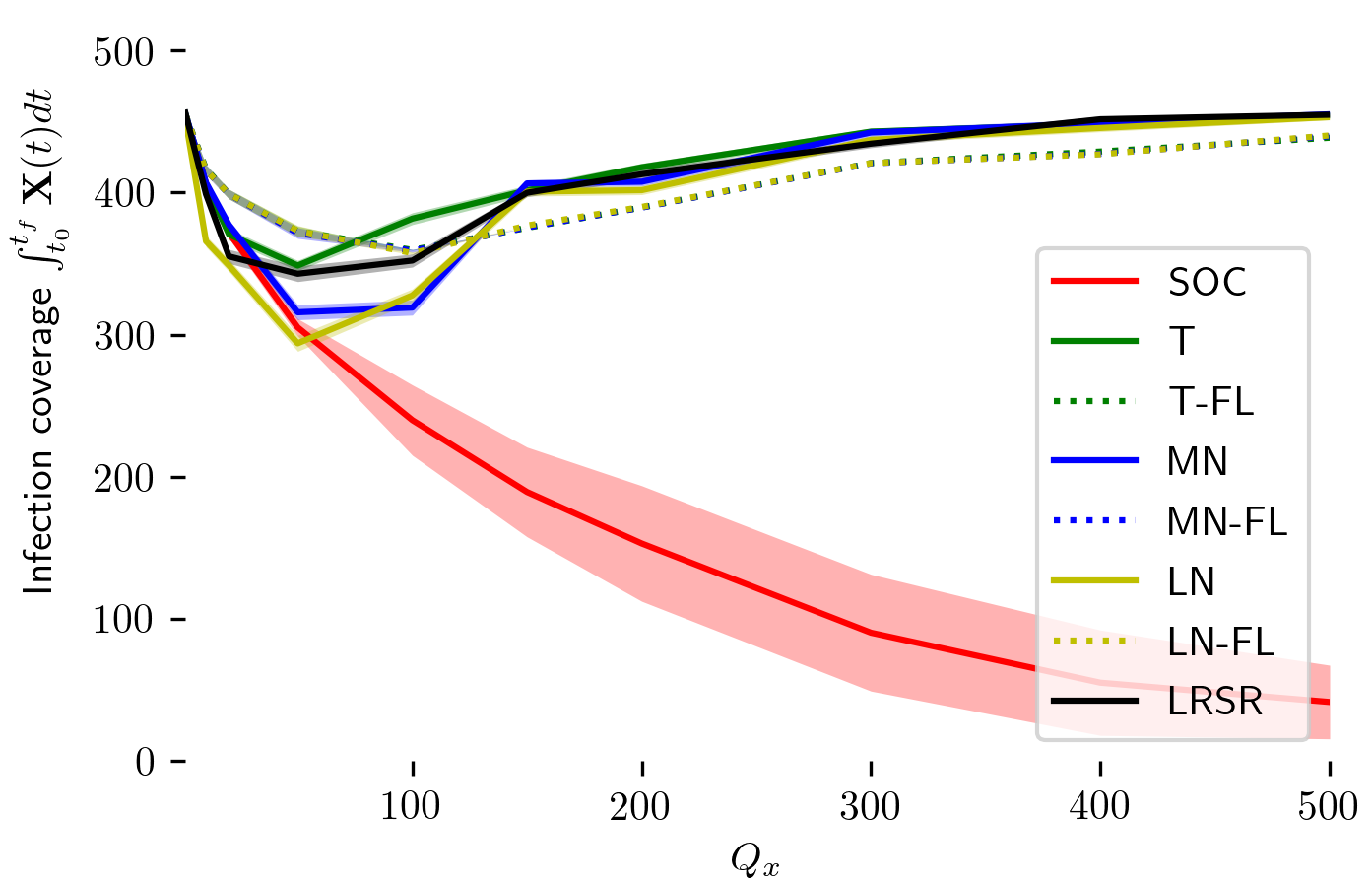}
\includegraphics[width=0.305\textwidth]{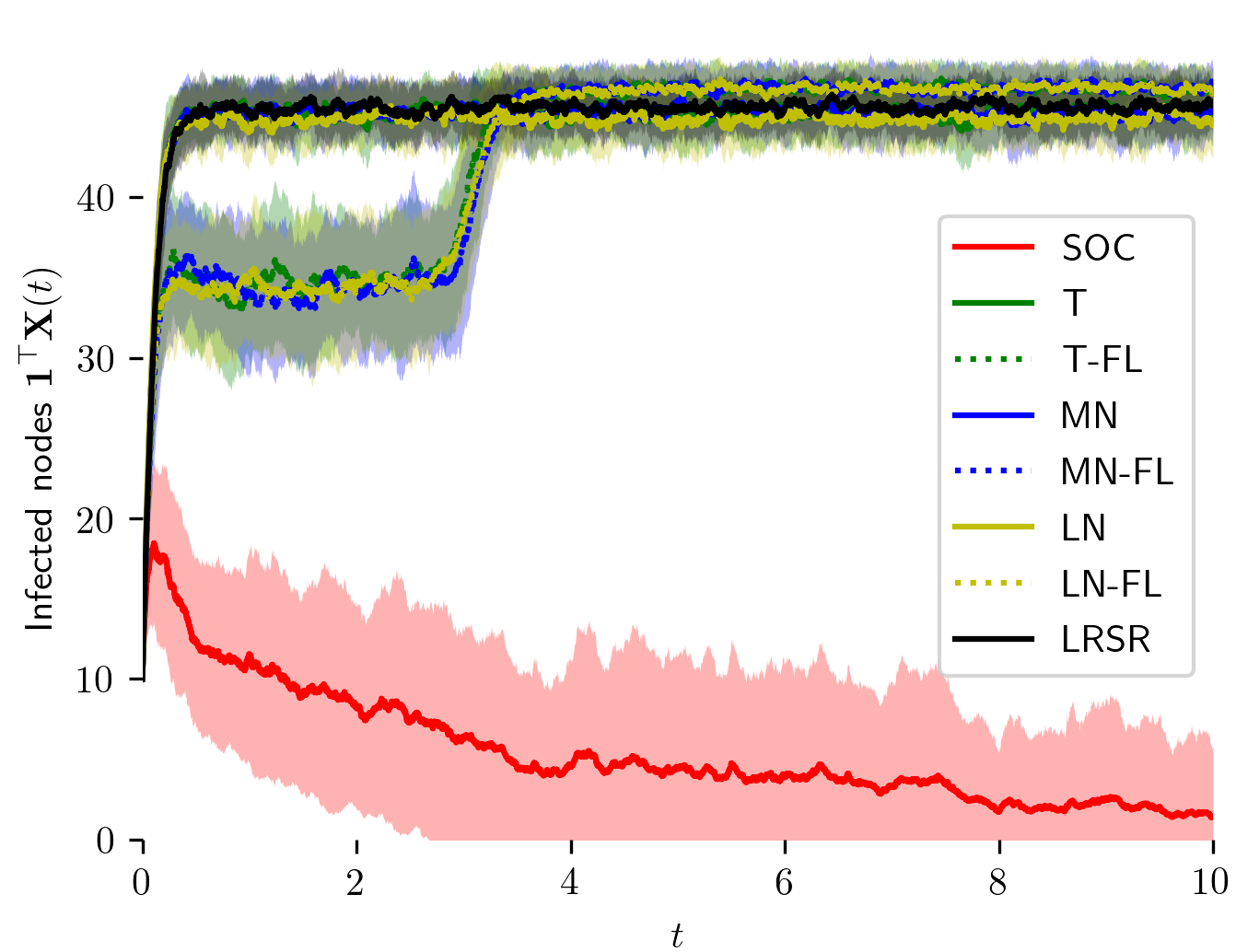}
\includegraphics[width=0.33\textwidth]{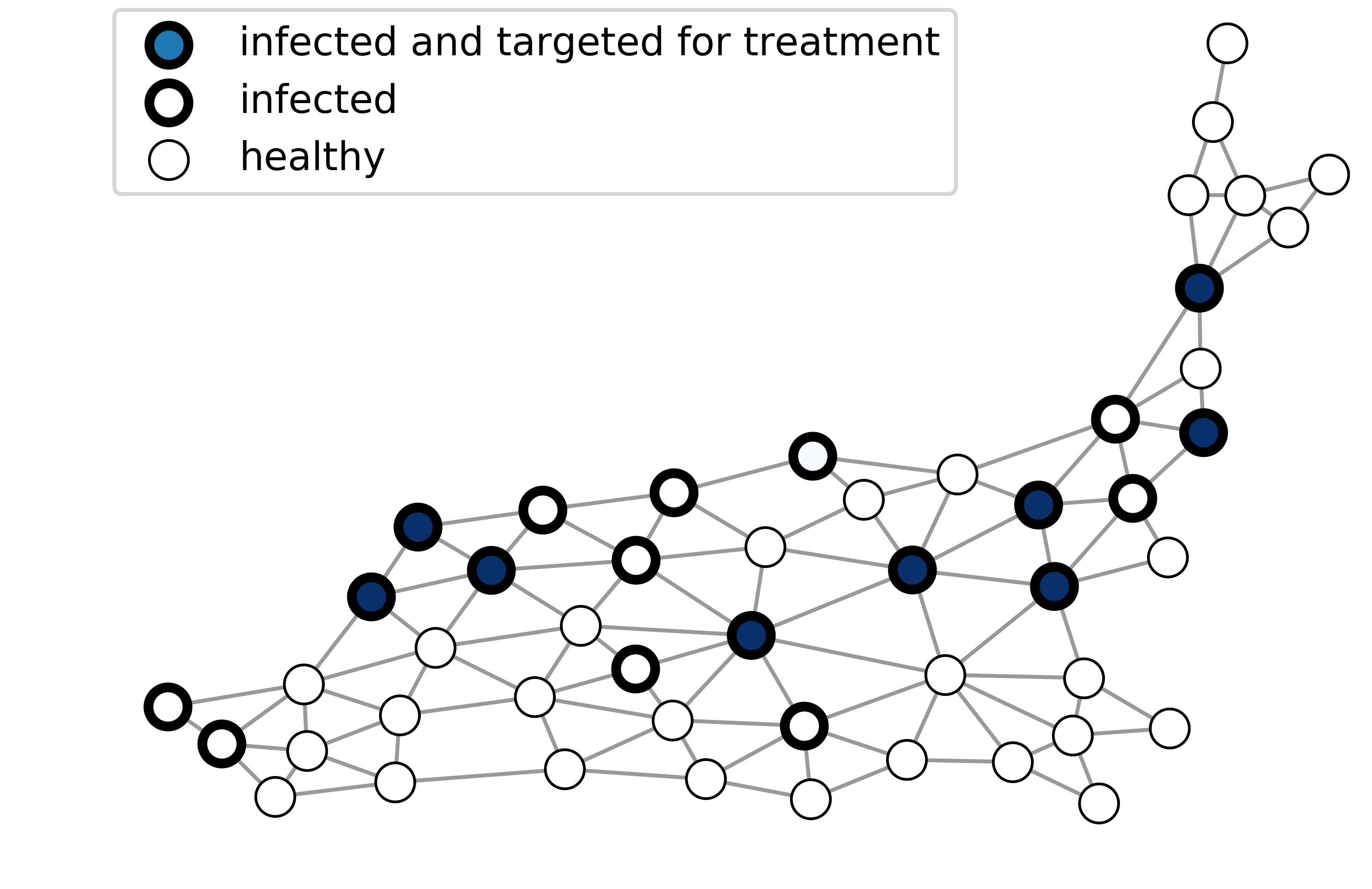}
\vspace{-3mm}
\caption{Performance of our stochastic optimal control algorithm [SOC] and the baselines. 
The left panel shows the total infection coverage $\int_{t_0}^{t_f} \mathbf{X}(t) dt$ against the parameter $\qb_X$ in our algorithm, where the
shaded regions indicate standard error of the mean.
The centered panel shows the number of infected nodes $\mathbf{1}^{\top} \mathbf{X}(t)$ against time, where the shaded regions indicate the $95$\%
confidence interval and $\qb_X = 400$.
The right panel shows a snapshot of the infected and healthy nodes as well as the nodes undergoing treatment, as picked by our algorithm, at a
particular time.
In all panels, we used $\beta = 6$, $\gamma = 5$, $\delta = 1$, $\rho = 5$, $\eta = 1$ and $\Qb_\Lambda = 1$.}
\label{fig:results}
\vspace{-4mm}
\end{figure}


\vspace{-1mm}
\section{Conclusion}
\label{sec:conclusion}
\vspace{-1mm}

In this paper, we approach disease modeling and control from the perspective of marked temporal point processes (MTPPs)
and stochastic optimal control of stochastic differential equations (SDEs) with jumps.
This novel perspective allows us to determine optimal online micro-interventions---individual treatments---as time goes on. While we have
applied our methodology to susceptible-infected-susceptible (SIS) epidemics, we believe that our methodology has the potential to guide 
the development of data-driven control strategies of other types of epidemic processes as well.

\newpage

{
\vspace{-10mm}
\small
\bibliographystyle{abbrv}
\bibliography{refs}
\vspace{-10mm}
}

\newpage

\appendix

\section{Proof of Lemma~\ref{lemma:diff}} \label{app:lemma:diff}
\begin{proof}
Starting from the definition of the optimal cost-to-go, we can proceed as follows:
\begin{align*}
&J(\Xb(t), \Hb(t), \Zb(t), \Mb(t), t)  \\
&=  \underset{\Lambdab(t, \infty)}{\text{min}} \EE_{(\Yb, \Wb, \Nb)(t, \infty)}\left[ \int_{t}^{t + dt}e^{-\eta \tau} \ell( \Lambdab(\tau), \Xb(\tau), \Hb(\tau), \Zb(\tau)) d\tau \right. \\
&\qquad \left.+ \int_{t + dt}^{\infty}e^{-\eta \tau}  \ell(\Lambdab(\tau), \Xb(\tau), \Hb(\tau) ,\Zb(\tau)) d\tau  \right] \\
&= \underset{\Lambdab(t, t + dt]}{\text{min}} \underset{\Lambdab(t + dt, \infty)}{\text{min}} \EE_{(\Yb, \Wb, \Nb)(t, t + dt]} \Bigg[  e^{-\eta t} \ell( \Lambdab(t), \Xb(t), \Hb(t) ,\Zb(\tau)) dt \\
&\qquad +  \EE_{(\Yb, \Wb, \Nb)(t + dt, \infty)}\left[ \int_{t + dt}^{\infty}e^{-\eta \tau}  \ell( \Lambdab(\tau), \Xb(\tau), \Hb(\tau) ,\Zb(\tau)) d\tau  \right] \Bigg] \\
&=  \underset{\Lambdab(t, t+dt]}{\text{min}} \left\{ \EE_{(\Yb, \Wb, \Nb)(t, t+dt]}\Big[ e^{-\eta t} \ell( \Lambdab(t), \Xb(t), \Hb(t) ,\Zb(\tau))dt \right. \\
& \left. \qquad \qquad\qquad\qquad\qquad\qquad\quad + J(\Xb(t + dt), \Hb(t + dt), \Zb(t + dt), \Mb(t + dt), t + dt)  \Big] \right\}
\end{align*}
%
The above implies that the optimal cost-to-go function satisfies Bellman'{}s Principle of Optimality~\cite{hanson2007}.
Moreover, if we assume $J$ is left-continuous, we have $J(\Xb(t+dt), \Hb(t + dt), \Zb(t + dt), \Mb(t + dt), t+dt) = J(\Xb(t), \Hb(t), \Zb(t), \Mb(t), t) + dJ(\Xb(t), \Hb(t), \Zb(t), \Mb(t), t)$ and can rewrite the above equation as:
%
\begin{align}
0 = \underset{\Lambdab(t, t+dt]}{\text{min}} 
\Big\{ \EE_{(\Yb, \Wb, \Nb)(t, t+dt]}\left[ dJ(\Xb(t), \Hb(t), \Zb(t), \Mb(t), t) \right] + e^{-\eta t} \ell(\Lambdab(t), \Xb(t), \Hb(t))dt  \Big\}\label{eq:bellman}
\end{align}
\end{proof}
%
%
\section{An Auxiliary Lemma} \label{app:lemma:aux}
\begin{lemma} \label{lemma:aux}
 The differential ($dJ(.)$) of the cost-to-go function $J(\Xb(t), \Hb(t), \Zb(t), \Mb(t), t)$ in Eq.~\ref{eq:def-cost-to-go} is given by
 \begin{align}\label{Jdifferentialxx}
&\hspace*{-0.1cm}dJ(\Xb(t), \Hb(t), \Zb(t), \Mb(t), t) = J_t(\Xb(t), \Hb(t), \Zb(t), \Mb(t), t ) dt \nn \\
&\hspace*{-0.4cm}+ \sum_{i \in \Vcal} \left[ J(\Xb(t) + \eb_i, \Hb(t), \Zb(t) + \Ab^\top \eb_i, \Mb(t), t) - J(\Xb(t), \Hb(t), \Zb(t), \Mb(t), t) \right] dY_i(t) \nn \\
&\hspace*{-0.4cm}+ \sum_{i \in \Vcal} \left[ J(\Xb(t) - \eb_i, \Hb(t) - \Hb(t) \odot \eb_i, \Zb(t) - \Ab^\top \eb_i, \Mb(t) - \Ab^\top(\Hb(t) \odot \eb_i), t) \right.\nn\\
&\hspace*{-0.4cm}\left.\qquad\quad- J(\Xb(t), \Hb(t), \Zb(t), \Mb(t), t) \right] dW_i(t) \nn\\
&\hspace*{-0.4cm}+ \sum_{i \in \Vcal} \left[ J(\Xb(t), \Hb(t) + \eb_i, \Zb(t), \Mb(t) + \Ab^\top \eb_i, t) - J(\Xb(t), \Hb(t), \Zb(t), \Mb(t), t) \right] dN_i(t)
\end{align}
\end{lemma}
\begin{proof}
Now, we differentiate $J$ with respect to time $t$, $\Xb(t)$, $\Hb(t)$, $\Zb(t)$, and $\Mb(t)$  using It\^{o}'{}s calculus. By the definition of the derivative we have
%
\begin{align*}
&dJ(\Xb(t), \Hb(t), \Zb(t), \Mb(t), t)\\
&= J(\Xb(t) + d\Xb(t), \Hb(t) + d \Hb(t), \Zb(t) + d\Zb(t), \Mb(t) + d\Mb(t), t + dt) \\
&~~~~ - J(\Xb(t), \Hb(t), \Zb(t), \Mb(t), t)
\end{align*}
and substituting in the SDE dynamics of Eqs.~\ref{eq:sis-vars} we get
\begin{align*}
= \;  &J\Big(\Xb(t) + d\Yb(t) - d\Wb(t), \Hb(t) + d\Nb(t) - \Hb(t) \odot d\Wb(t),  \Zb(t) + \Ab^\top  d\Yb(t) - \Ab^\top d\Wb(t), \nn\\
  &\quad  \Mb(t) + \Ab^\top d \Nb(t) - \Ab^\top ( \Hb(t) \odot d\Wb(t)), t + dt \Big)  
- J \big(\Xb(t), \Hb(t), \Zb(t), \Mb(t), t\big)
\end{align*}
We now apply the zero-one jump law~\cite{hanson2007} to obtain 
%
\begin{align*}
& dJ(\Xb(t), \Hb(t), \Zb(t), \Mb(t), t) \\
&=  \sum_{i \in \Vcal} J(\Xb(t) + \eb_i, \Hb(t), \Zb(t) + \Ab^\top \eb_i, \Mb(t), t + dt)  dY_i(t) \\
&\quad + \sum_{i \in \Vcal} J(\Xb(t) - \eb_i, \Hb(t) - \Hb(t) \odot \eb_i, \Zb(t) - \Ab^\top \eb_i, \Mb(t) - \Ab^\top ( \Hb(t) \odot \eb_i), t + dt) dW_i(t)\\
&\quad+ \sum_{i \in \Vcal} J(\Xb(t), \Hb(t) + \eb_i, \Zb(t), \Mb(t) + \Ab^\top \eb_i,  t + dt) dN_i(t) \\
&\quad+ J(\Xb(t), \Hb(t), \Zb(t), \Mb(t), t + dt) \left[ 1 - \sum_{i \in \Vcal} (dY_i(t) + dW_i(t) + dN_i(t))\right] \\
&\quad- J(\Xb(t), \Hb(t), \Zb(t), \Mb(t), t)
\end{align*}
where we used $\forall i \in \Vcal:  dY_i dW_i = 0, dY_i dN_i = 0, dW_i dN_i = 0, dY_i dW_i dN_i = 0$ by the zero-one jump law \cite{hanson2007}. 

We simplify further, and as an intermediate step, we use the total derivative rule, \ie $J(., t + dt) =  J(., t)  +  J_t(., t)dt$, to rewrite the expression above. 

Using that the bilinear differential form is $dtdY_i = dtdW_i = dtdN_i = 0$ \cite{hanson2007}, we finally obtain the differential for $J$ as it is stated in the lemma.
\end{proof}
%
%
%
\section{Proof of Lemma~\ref{lemma:hjb}} \label{app:lemma:hjb}
\begin{proof}
First, we substitute Eq.~\ref{Jdifferentialxx}, given by auxiliary Lemma~\ref{lemma:aux} (refer to Appendix~\ref{app:lemma:aux}), into 
Eq.~\ref{eq:bellman} of Lemma \ref{lemma:diff} (refer to Appendix~\ref{app:lemma:diff}).
Note that only the last two terms in the minimization over $\Lambdab(t, t+ dt]$ depend on the control intensity $\Lambdab(t)$. Hence, we have
\begin{align}\label{eq:cost-to-go-J-hjb}
\begin{split}
0 &= J_t(\Xb(t), \Hb(t), \Zb(t), t ) dt   \\
&\textstyle \quad+ \sum_{i \in \Vcal} \left[ J(\Xb(t) + \eb_i, \Hb(t), \Zb(t) + \Ab^\top \eb_i, \Mb(t), t) \right. \\
&\textstyle\qquad \qquad \quad \left. - J(\Xb(t), \Hb(t), \Zb(t), \Mb(t), t) \right] (1-X_i(t)) \beta Z_i(t) dt  \\
&\textstyle\quad+ \sum_{i \in \Vcal} \left[ J(\Xb(t) - \eb_i, \Hb(t) - \Hb(t) \odot \eb_i, \Zb(t) - \Ab^\top \eb_i, \Mb(t) - \Ab^\top( \Hb(t) \odot \eb_i), t) \right. \\
&\textstyle \qquad \qquad \quad \left. - J(\Xb(t), \Hb(t), \Zb(t), \Mb(t), t) \right] X_i(t) ( \delta + \rho H_i(t) ) dt \\
&\textstyle\quad+  \underset{\Lambdab(t, t+dt]}{\text{min}} \Big\{ \sum_{i \in \Vcal} \left[ J(\Xb(t), \Hb(t) + \eb_i, \Zb(t), \Mb(t) + \Ab^\top\eb_i, t) \right. \\
&\textstyle \left. \qquad\qquad\qquad\quad \quad \quad - \; J(\Xb(t), \Hb(t), \Zb(t), \Mb(t), t) \right] \lambda_i (t) X_i(t) (1 - H_i(t)) dt \\
&\textstyle\qquad \qquad \qquad \quad + \; e^{-\eta t}\ell(\Lambdab(t), \Xb(t), \Hb(t))dt  \Big\} 
\end{split}
\end{align}
Since we optimize over an \emph{infinite} time horizon, the optimal cost can be expressed in present value terms, independent of $t$. Recall the definition of the optimal \emph{cost-to-go} function $J$ from Eq. \ref{eq:def-cost-to-go}
\begin{align}
\begin{split}\label{eq:def-cost-function}
&J(\Xb(t), \Hb(t), \Zb(t), \Mb(t), t) \\
&= \underset{\Lambdab(t, \infty)}{\text{min}} \EE_{(\Yb, \Wb, \Nb)(t, \infty)}\left[ \int_{t}^{\infty} e^{-\eta \tau} \; \ell(\Lambdab(\tau), \Xb(\tau), \Hb(\tau), \Zb(\tau)) d\tau \right]  \\
&= e^{-\eta t} \, \underset{\Lambdab(t, \infty)}{\text{min}} \EE_{(\Yb, \Wb, \Nb)(t, \infty)}\left[ \int_{t}^{\infty} e^{-\eta (\tau - t)} \; \ell(\Lambdab(\tau), \Xb(\tau), \Hb(\tau), \Zb(\tau)) d\tau \right]  \\[3pt]
&= e^{-\eta t} \,\,\, V(\Xb(t), \Hb(t), \Zb(t), \Mb(t))
\end{split}
\end{align}
where 
$$ \hspace*{-0.1cm}V(\Xb(t), \Hb(t), \Zb(t), \Mb(t)) =  \hspace*{-0.1cm} \underset{\Lambdab(t, \infty)}{\text{min}} \EE_{(\Yb, \Wb, \Nb)(t, \infty)}\left[ \int_{t}^{\infty} \hspace*{-0.2cm} e^{-\eta (\tau - t)} \ell(\Lambdab(\tau), \Xb(\tau), \Hb(\tau), \Zb(\tau)) d\tau \right]$$ now is the time-independent cost function. The value of the integral depends on the initial states, not on the initial time, i.e. only on the elapsed time \cite{kamienschwartz}. We substitute Eq. \ref{eq:def-cost-function} into Eq. \ref{eq:cost-to-go-J-hjb} and use the fact that $J_t(.) = - \eta e^{- \eta t} V(.)$. Multiplying both sides by $e^{\eta t}$ finally yields the Hamilton-Jacobi-Bellman (HJB) equation of stochastic dynamic programming
\begin{align}
\begin{split}\label{eq:hjb}
0 &= - \eta V(\Xb(t), \Hb(t), \Zb(t), \Mb(t)) dt   \\
&\textstyle \quad+ \sum_{i \in \Vcal} \left[ V(\Xb(t) + \eb_i, \Hb(t), \Zb(t) + \Ab^\top \eb_i, \Mb(t)) \right. \\
&\left.\qquad \qquad \quad- V(\Xb(t), \Hb(t), \Zb(t), \Mb(t)) \right] (1-X_i(t)) \beta Z_i(t) dt  \\
&\textstyle\quad+ \sum_{i \in \Vcal} \left[ V(\Xb(t) - \eb_i, \Hb(t) - \Hb(t) \odot \eb_i, \Zb(t) - \Ab^\top \eb_i, \Mb(t) - \Ab^\top( \Hb(t) \odot \eb_i),) \right. \\
&\textstyle \qquad \qquad \quad \left. - V(\Xb(t), \Hb(t), \Zb(t), \Mb(t)) \right] X_i(t) ( \delta + \rho H_i(t) ) dt \\
&\textstyle\quad+  \underset{\Lambdab(t, t+dt]}{\text{min}} \Big\{ \sum_{i \in \Vcal} \left[ V(\Xb(t), \Hb(t) + \eb_i, \Zb(t), \Mb(t) + \Ab^\top\eb_i) \right. \\
&\textstyle \left. \qquad\qquad\qquad\quad \quad \quad - \; V(\Xb(t), \Hb(t), \Zb(t), \Mb(t)) \right] \lambda_i (t) X_i(t) (1 - H_i(t)) dt \\
&\textstyle\qquad \qquad \qquad \quad + \; \ell(\Lambdab(t), \Xb(t), \Hb(t))dt  \Big\} 
\end{split}
\end{align}
%
%
Now, for further notational simplicity, we define the diagonal matrices:
\begin{align*}
&\big(\Delta_{V}^Y(\Xb(t), \Hb(t), \Zb(t), \Mb(t)) \big)_{ii} \\
&=~ V(\Xb(t) + \eb_i,\Hb(t), \Zb(t) + \Ab^\top \eb_i, \Mb(t), t) - V(\Xb(t), \Hb(t),\Zb(t), \Mb(t)) \\[2mm]
&\big(\Delta_{V}^W(\Xb(t), \Hb(t), \Zb(t), \Mb(t)) \big)_{ii} \\
&=~ V(\Xb(t) - \eb_i, \Hb(t) - \Hb(t) \odot \eb_i, \Zb(t) - \Ab^\top \eb_i, \Mb(t) - \Ab^\top ( \Hb(t) \odot \eb_i)) \\
& ~~~~~- V(\Xb(t), \Hb(t),\Zb(t), \Mb(t)) \\[2mm]
&\big(\Delta_{V}^N(\Xb(t), \Hb(t),\Zb(t), \Mb(t)) \big)_{ii} \\
&=~ V(\Xb(t), \Hb(t) + \eb_i, \Zb(t), \Mb(t) + \Ab^\top \eb_i) - V(\Xb(t), \Hb(t),\Zb(t), \Mb(t)) 
\end{align*}
and $\forall i,j$ s.t. $i \neq j$, $\big(\Delta_{V}^Y (., ., ., .)\big)_{ij} = \big(\Delta_{V}^W (., ., ., .)\big)_{ij} = \big(\Delta_{V}^N (., ., ., .)\big)_{ij} = 0$. From now on, we will omit the arguments of $\Delta_{V}^Y$, $\Delta_{V}^W $, and $\Delta_{V}^N$.
Under these definitions, if we use the functional form of $\ell(\Lambdab(t), \Xb(t))$ in Eq.~\ref{eq:loss}, the control signal $\Lambdab(t)$ can be obtained by solving the minimization in the HJB equation
\begin{align}
\underset{\Lambdab(t_0, \infty)}{\text{minimize}} & \quad \Lambdab^\top(t) \,\diag \big(\mathbf{1} - \Hb(t) \big) \, \Delta_{V}^N \Xb(t)  + \tfrac{1}{2} \Lambdab^\top(t) \Qb_{\Lambdab} \Lambdab(t) + \qb_{\Xb} \Xb(t)  \nonumber \\
\text{subject to} & \quad \lambda_i(t) \geq 0 \quad \forall i \in \Vcal, \forall t \in (t_0, \infty), \label{eq:hjb-minimization}
\end{align}

By differentiating Eq.~\ref{eq:hjb-minimization}, with respect to $\Lambdab(t)$, the solution to the \emph{unconstrained} minimization is given by:
\begin{equation}\label{eq:lambda-sol-unconstr}
\Lambdab^*(t) = - \Qb_{\Lambdab}^{-1} \diag \big(\mathbf{1} - \Hb(t) \big) \, \Delta_{V}^N  \, \Xb(t) 
\end{equation}
%
which is the same as the solution to the constrained problem since we find a valid solution with $(\Delta_{V}^N (\Xb(t), \Hb(t), \Zb(t), \Mb(t)))_{ii} \leq 0$ in Theorem \ref{thm:finalxx}, see appendix \ref{app:thm:finalxx}, 
$X_i(t) \geq 0$ as $X_i(t) \in \{0,1\}$, $1 - H_i(t) \geq 0$ as $H_i(t) \in \{ 0,1\}$,  
and $(\Qb_{\Lambdab})_{ij} \geq 0$ by definition.

Finally, if we combine Eqs.~\ref{eq:hjb} and~\ref{eq:lambda-sol-unconstr}, we find that the optimal cost function $V$ needs to satisfy Eq.~\ref{eq:pde}, which concludes
the proof.
\end{proof}
\section{Proof of Theorem~\ref{thm:finalxx}} \label{app:thm:finalxx}
\begin{proof}
Assume the functional form of the optimal cost function $V$ is given by:
\begin{align}
\begin{split}
V(\Xb(t), \Hb(t), \Zb(t), \Mb(t)) &= \bb^\top \Xb(t) + \cb^\top \Hb(t) + \db^\top \Zb(t) + \fb^\top \Mb(t)
\end{split}
\end{align}
where $\bb, \cb, \db, \fb$ are constants with respect to time as $V$ is independent of time $t$. Then, given this functional form, we obtain the following explicit expressions for the jump terms
\begin{align*}
\Delta^Y_V &= \diag \big( \bb + \Ab \db ) \\
\Delta^W_V &= \diag \big(- \bb -\Hb(t) \odot  \cb - \Ab \, \db - \Hb(t) \odot \Ab\fb \big) \\
\Delta^N_V &= \diag \big(\cb + \Ab \fb \big)
\end{align*}
Using Eq. (\ref{eq:pde}), the optimal cost function $V$ has to satisfy
\begin{align}\label{eq:pde-sub}
\begin{split}
0 &= - \eta\bb^\top \Xb(t)  - \eta \cb^\top \Hb(t) - \eta\db^\top \Zb(t) - \eta \fb^\top \Mb(t)  \\
&\quad+  (\mathbf{1}-\Xb(t))^\top\diag \big( \bb + \Ab \db )  \, (\beta\Zb(t) - \gamma  \Mb(t) )\\
&\quad+ ( \delta \mathbf{1} + \rho\Hb)^\top  \diag \big(- \bb -\Hb(t) \odot  \cb - \Ab \, \db - \Hb(t) \odot \Ab \fb \big)   \,\Xb(t) \\
&\quad- \tfrac{1}{2} \Xb(t)^\top \big( \diag \big(\cb + \Ab \fb \big) \big)^2   \diag\big( \mathbf{1} - \Hb(t) \big)^2 \Qb_{\Lambdab}^{-1}  \: \Xb(t)  + \qb_{\Xb} \Xb(t)  
\end{split}
\end{align}
To find the unknown constant $\bb, \cb, \db, \fb$ such that $V$ satisfies the PDE, we equate the coefficients of the different variables $\Xb(t), \Hb(t), \Zb(t), \Mb(t)$ with zero. We use the fact that in the domain of $V$, we have $\Xb(t)^\top \Xb(t) = \mathbf{1}^\top \Xb(t) $ and $\Hb(t)^\top \Hb(t) =\mathbf{1}^\top \Hb(t)$ since $X_i(t), H_i(t) \in \{0,1\}$. Likewise, $H_i(t) = 1 \Rightarrow X_i(t) = 1$, hence $\Xb(t)^\top \Hb(t)  = \mathbf{1}^\top\Hb(t)$. Thus, rewriting the PDE, we obtain
\begin{subequations}
\begin{align}
0 &=  \Big( - \eta \bb + \delta(-\bb - \Ab \db) - \tfrac{1}{2} \mathbf{1}^\top \Qb_\Lambda^{-1} (\cb + \Ab \fb)^2 + \qb_X  \Big) \Xb(t)\label{eq:pde-factored-X} \\ 
& +  \Big ( - \eta \db + \beta ( \bb + \Ab \db ) \Big ) \Zb(t) 
+ \Xb(t)^\top \Big (- \beta( \bb + \Ab \db)  \Big )\Zb(t)  \label{eq:pde-factored-Z} \\
& +  \Big ( - \eta \fb - \gamma ( \bb + \Ab \db ) \Big ) \Mb(t) 
+ \Xb(t)^\top \Big ( \gamma ( \bb + \Ab \db)  \Big )\Mb(t)  \label{eq:pde-factored-M} \\
& + \Hb(t)^\top \Big ( - \eta \diag (\cb) -(\delta+ \rho) \diag(\cb + \Ab \fb) \nn\\
&\qquad \qquad \quad - (\delta+ \rho) \diag (\bb + \Ab \db) + \tfrac{1}{2} \Qb_\Lambda^{-1} \diag(\cb + \Ab \fb)^2   \Big ) \Xb(t) \label{eq:pde-factored-HX}
\end{align}
\end{subequations}
 To satisfy the PDE, all coefficients must be zero. To make the $6 |\Vcal|$ equations solvable with $4 |\Vcal|$ variables and ensure $\Delta_V^N =  \diag(\cb + \Ab \fb) \leq 0$ to satisfy the constrained optimization problem of Eq. \ref{eq:hjb-minimization}, we take advantage of the fact that $\Xb \in \{0,1\}^{|\Vcal|}$ to distinguish two cases for every indicator variable: $\Xb_i(t) = 0$ and $\Xb_i(t) = 1$. Given either condition, terms of the PDE group differently and we get the following system of equations for the $i$-th indices of the constants:

\emph{If} $\Xb_i(t) = 0$:
\begin{subequations}
\begin{align}
 - \eta \db_i + \beta ( \bb_i + \Ab_{i*} \db ) & = 0 \label{eq:cases-X0-a} \\
 \db_i + \tfrac{\beta}{\gamma} \fb_i & = 0  \label{eq:cases-X0-b} 
\end{align}
\end{subequations}
\emph{If} $\Xb_i(t) = 1$:
\begin{subequations}
\begin{align}
 - \eta \db_i & = 0 \label{eq:cases-X1-a} \\
 - \eta \fb_i & = 0 \label{eq:cases-X1-b} \\
 - \eta \cb_i - (\delta+ \rho) (\cb_i + \Ab_{i*} \fb) - (\delta+ \rho) (\bb_i + \Ab_{i*} \db) + \tfrac{1}{2} \Qb_\Lambda^{-1} (\cb_i + \Ab_{i*} \fb)^2 & = 0  \label{eq:cases-X1-c} \\
 - \eta \bb_i - \delta(\bb_i + \Ab_{i*} \db) - \tfrac{1}{2} \Qb_{\Lambda,i}^{-1} (\cb_i + \Ab_{i*} \fb)^2 + \qb_{X,i}  & = 0  \label{eq:cases-X1-d}
\end{align}
\end{subequations}
First off, note that if $\Xb_i(t) = 0$, there are \emph{no} equations determining $\cb_i$, and only $(\Delta_V^N)_{ii} = \cb_i + \Ab_{i*} \fb \leq 0$ has to hold from the optimization problem of Eq. \ref{eq:hjb-minimization}. This is fine as $\Xb_i(t) = 0$ implies $\Hb_i(t) = 0$, and $\cb_i$ appears as a factor of $\Hb_i(t)$ both in the optimal cost function $V$ and the optimal intensity $\Lambdab^*$. In other words, the optimal intensity will be zero for node $i$ anyways if $\Xb_i(t) = 0$. Hence, to ensure the constraint, let $\cb_i = - \Ab_{i*} \fb$ if $\Xb_i(t) = 0$.

From Eqs. \ref{eq:cases-X1-a} and \ref{eq:cases-X1-b} we know that if $\Xb_i(t) = 1$, then $\db_i = 0$ and $\fb_i = 0$. Thus, $\db_i$ and $\fb_i$ are only non-zero if $\Xb_i(t) = 0$ and we can rewrite Eqs. \ref{eq:cases-X0-a} and \ref{eq:cases-X0-b} as
\begin{align}
 \bb_{(0)} & = \left ( \tfrac{\eta}{\beta} \Ib - \Ab_{(00)} \right) \db_{(0)} \label{eq:cases-b-0}\\
\fb_{(0)} & = -\tfrac{\gamma}{\beta}  \db_{(0)} \label{eq:cases-f-0}
\end{align}
where $\Ab_{(00)}$ is the matrix that contains the $i$-th rows and $i$-th columns of the adjacency matrix $\Ab$ if and only if $\Xb_i(t) = 0$. Thus, $\Ab_{(00)} \in \mathbb{R}^{m \times m}$ where $m = |\Vcal| - |\Xb(t) |_1$. Likewise,  $ \bb_{(0)} $, $\db_{(0)}$, $\fb_{(0)}$ are the vectors $\bb$, $\db$, $\fb$ containing only indices where $\Xb_i(t) = 0$, and we know $\db_{(1)} = 0$, $\fb_{(1)} = 0$. The same notation is used when rewriting the two remaining Eqs. \ref{eq:cases-X1-c} and \ref{eq:cases-X1-d} in the following, already using Eqs. \ref{eq:cases-X1-a}, \ref{eq:cases-X1-b}, \ref{eq:cases-f-0}
\begin{align}
 -\eta \cb_{(1)} - (\delta + \rho) \left(\cb_{(1)} -\tfrac{\gamma}{\beta} \Ab_{(10)} \db_{(0)} \right)  - (\delta + \rho) \left(\bb_{(1)} + \Ab_{(10)} \db_{(0)} \right) &\nn\\
 + \tfrac{1}{2}\Qb_{\Lambda, (11)}^{-1} \left (\cb_{(1)} -\tfrac{\gamma}{\beta} \Ab_{(10)} \db_{(0)} \right )^2 & = 0 \label{eq:cases-bc-1}\\
- \eta \bb_{(1)} - \delta ( \bb_{(1)} + \Ab_{(10)} \db_{(0)}) - \tfrac{1}{2}\Qb_{\Lambda, (11)}^{-1} \left (\cb_{(1)} -\tfrac{\gamma}{\beta} \Ab_{(10)} \db_{(0)} \right )^2 + \qb_{X, (1)} & = 0\label{eq:cases-bc-2}
\end{align}
where $()^2$ is the element-wise square of the vector. $\Ab_{(10)}$ is the matrix that contains the $i$-th rows of $\Ab$ iff $\Xb_i(t) = 1$ and $i$-th columns of $\Ab$ iff $\Xb_i(t) = 0$. Now, all constants can be determined to find the proposal $V$ that satisfies the HJB; however, note that the above system of equations is overdetermined. This additional degree of freedom is necessary to ensure that $\Delta_V^N = \diag ( \cb + \Ab \fb ) \leq 0$ for all possible system parameters.

Thus, we express all constants as expressions of $\db_{(0)}$ and choose one particular solution of our solution space which is optimal with respect to our loss function. Specifically, \emph{we choose the policy that achieves the optimum under minimal treatment interventions in the population}.

Eqs. \ref{eq:cases-bc-1} and \ref{eq:cases-bc-2} can be solved in terms of $\bb_{(1)}$ and $\cb_{(1)}$ in closed-form. Ultimately, we obtain the following expression for the term $\Delta^N_{V, (1)} = \diag (  \cb_{(1)} + \Ab_{(10)} \fb_{(0)}  )$ which determines the optimal policy $\Lambdab^*(t)$:
\begin{align}
&\cb_{(1)} + \Ab_{(10)} \fb_{(0)}  = \tfrac{1}{\beta (2 \delta + \eta + \rho)} \Bigg ( \beta(\delta + \eta) ( \delta + \eta + \rho) \Qb_{\Lambda} \mathbf{1} \\
&\hspace*{-0.3cm}\pm \sqrt{ 2 \beta \Qb_\Lambda (2 \delta \hspace*{-0.05cm} +\hspace*{-0.05cm} \eta \hspace*{-0.05cm}+\hspace*{-0.05cm} \rho)\Big [ \eta \Ab_{(10)} \db_{(0)} (\gamma (\delta \hspace*{-0.05cm} +\hspace*{-0.05cm}  \eta) \hspace*{-0.05cm} +\hspace*{-0.05cm}  \beta (\delta \hspace*{-0.05cm} + \hspace*{-0.05cm} \rho)) \hspace*{-0.05cm} +\hspace*{-0.05cm}  \beta (\delta \hspace*{-0.05cm} +\hspace*{-0.05cm}  \rho) \qb_X \Big]\hspace*{-0.05cm} + \hspace*{-0.05cm} \beta^2 (\delta \hspace*{-0.05cm} + \hspace*{-0.05cm} \eta)^2 (\delta\hspace*{-0.05cm}  + \hspace*{-0.05cm} \eta \hspace*{-0.05cm} +\hspace*{-0.05cm}  \rho)^2 \Qb_\Lambda^2  \mathbf{1} } \Bigg ) \nonumber
\end{align}
where the subscripts of $\Qb_{\Lambda, (11)}, \qb_{X, (1)}$ are implied. Clearly, we choose the negative solution of $\cb_{(1)}$ to Eqs. \ref{eq:cases-bc-1} and \ref{eq:cases-bc-2}. Then, as long as the expression under the radical is positive, we have a feasible solution to our constrained optimization problem of Eq. \ref{eq:hjb-minimization} and found a closed-form expression for the policy. $\Delta^N_{V, (0)}$ is not important as it is a factor of $\Xb(t)$ with $\Xb_i(t) = 0$ at those entries and will thus be zero in all cases. Note that here $\Delta^N_{V, (0)} = 0$ since we chose $\cb_{(0)} = - \Ab_{(00)} \fb_{(0)}$ above.

It remains to show how we choose the \emph{policy that achieves the optimum under minimal treatment interventions in the population}. The positivity condition of the radical can be written as
\begin{align}
\Ab_{(10)} \db_{(0)} \geq - \tfrac{\beta (\delta + \rho)}{ \eta (\gamma (\delta + \eta) + \beta (\delta + \rho))} \qb_{X,(1)} \label{eq:radical-condition}
\end{align}
If Eq. \ref{eq:radical-condition} holds, then the radical expression is real and we have a valid optimal policy $\Lambdab^*$. Moreover, if the inequality is \emph{binding}, then $\Delta^N_{V, (1)} = 0$ and thus $\Lambdab^* = 0$, i.e. the policy doesn't intervene. We thus aim to satisfy the above constraints and principle of optimality \emph{achieving minimal interventions}. This means, at every time $t$ that $\Xb(t)$ changes, the optimal intensity $\Lambdab^*(t)$ is recomputed online by solving the following optimization problem
\begin{align}
\min_{\db_{(0)}} \quad& | \Ab_{(10)} \db_{(0)} - k \cdot \qb_{X,(1)}  |_1 \nonumber \\
\text{s.t.}\quad& \Ab_{(10)} \db_{(0)} \geq k \cdot \qb_{X,(1)} \label{eq:policy-optimization}
\end{align}
where $k = - \tfrac{\beta (\delta + \rho)}{ \eta (\gamma (\delta + \eta) + \beta (\delta + \rho))}$. This can be formulated as an LP as $
\min_{\zb, \db_{(0)}} \mathbf{1}^\top \zb  $, such that 
$ \zb = \Ab_{(10)} \db_{(0)} - k \cdot \qb_{X,(1)}$ and 
$\Ab_{(10)} \db_{(0)} \geq k \cdot \qb_{X,(1)} $, and common optimization algorithms apply.

The cost-to-go $V$ satisfies the PDE from Eq. \ref{eq:pde} with the found constants $\bb, \cb, \db, \fb$, which concludes the proof.
\end{proof}

%
%

\newpage

\section{Online stochastic optimal control algorithm} \label{app:algorithm}

In our experiments using synthetic data, every time that any of the state variables change, we update the treatment, infection, and recovery intensities. Then, we determine the next event by first sampling its time and subsequently the node and type it corresponds to. 
Algorithm~\ref{alg:sampling} summarizes our procedure.

\begin{algorithm}
  \caption{Online stochastic optimal control on synthetic data} \label{alg:sampling}
    \hspace*{\algorithmicindent} \textbf{Input:} network $\Ab$, epidemic parameters $\beta, \gamma, \rho, \delta$, control parameters $\Qb_\Lambda, \qb_X, \eta$;\\
    \hspace*{\algorithmicindent}~~~~~~~~~~~~ state variables $\Xb(t), \Hb(t)$\\
    \hspace*{\algorithmicindent} \textbf{Output:} time until next event $t^*$, node of next event $u^*$ 
    \begin{algorithmic}[1]
 \State Initialize $\db, \Lambdab^*(t) \gets \mathbf{0}$
\State $K^{(\text{I})} \gets \beta (2\delta + \eta + \rho) $
\State $K^{(\text{II})}_{\Lambda, i} \gets \beta (\delta + \eta) (\delta + \eta + \rho)  \Qb_{\Lambda, i} $
\State $K^{(\text{III})} \gets \eta ( \gamma ( \delta + \eta ) + \beta ( \delta + \rho ) ) $
\State $K^{(\text{IV})}_{X, i} \gets \beta (\delta + \rho) \qb_{X, i}$ 
\State $K^{(\text{V})} \gets  \gamma \eta + \beta ( 2\delta + \rho)$ 
\State Find $\db_{(0)} \coloneqq \db[X_i(t) = 0]$ by optimization of
\Statex$\min_{\db_{(0)}} \quad \left| \Ab_{(10)} \db_{(0)} + \tfrac{1}{K^{(\text{III})} } \bm{K}^{(\text{IV})}_{X, (1)}  \right|_1 \quad \text{s.t.}\quad \Ab_{(10)} \db_{(0)} \geq - \tfrac{1}{K^{(\text{III})} } \bm{K}^{(\text{IV})}_{X, (1)} $

%
%
%

\For{$i \in \Vcal$}
	            \State $\lambda_{N_i}(t) \gets - \tfrac{X_i(t)(1-H_i(t))}{K^{(\text{I})}} \Qb_{\Lambda, i}^{-1} \Big [ K^{(\text{II})}_{\Lambda, i}  - \sqrt{  2 K^{(\text{I})} \Qb_{\Lambda, i} \big ( K^{(\text{III})} \Ab_{i*} \db  + K^{(\text{IV})}_{X, i} \big )+ \big(K^{(\text{II})}_{\Lambda, i}\big)^2 } \Big ] $
		\State $\lambda_{Y_i}(t) = (1-X_i(t))(\beta \Ab_{*i}^\top \Xb(t) - \gamma \Ab_{*i}^\top \Hb(t))$
		\State $\lambda_{W_i}(t) = X_i(t) ( \delta + \rho H_i(t))$
\EndFor

\State $\Lambda_{\text{total}} \gets \sum_{i \in \Vcal} (\lambda_{N_i} + \lambda_{Y_i} + \lambda_{W_i})$

\State $u \sim \text{Unif}(0, 1)$
\State $t^* \gets - \log(u) / \Lambda_{\text{total}} $ \Comment{time until next event $\sim \text{Expo}(\Lambda_{\text{total}})$ }
\State $u^* \sim \text{Cat}([\lambda_{N_i} / \Lambda_{\text{total}}]_{i \in \Vcal}, [\lambda_{Y_i} / \Lambda_{\text{total}}]_{i \in \Vcal}, [\lambda_{W_i} / \Lambda_{\text{total}}]_{i \in \Vcal})$  \Comment{sample node and type of event}

\State \Return $t^*$, $u^*$ 

    \end{algorithmic}
  \label{algo:onlinealg}
\end{algorithm}

In a real world setting, infection and recovery events $\Yb(t)$ and $\Wb(t)$ occur independently of the sampling procedure for treatments $\Nb(t)$, yet 
they change its intensities $\mathbb{E}[d\Nb(t) | \Hcal(t) ]$. Thus, if an infection or recovery event occurs before the next sampled treatment event, the
treatment intensities change. It is then necessary to update the time of the planned treatment event using rejection sampling (if the treatment intensities
decrease) or the superposition principle (if the treatment intensities increase).
%
%
%

    
%
%
%

\end{document}